\documentclass[a4paper,10pt]{amsart}

\usepackage{a4wide}
\usepackage{amsthm}                   
\usepackage{amsmath}                  
\usepackage{amsfonts}                 
\usepackage{amssymb}                  
\usepackage{mathrsfs}                 
\usepackage{mathtools}								
\usepackage{pstricks-add}
\usepackage{xy}
\usepackage{hyperref,amsbsy}

\xyoption{all}

\theoremstyle{definition}
\newtheorem{definition}{Definition}[section]
\newtheorem{remark}[definition]{Remark}
\newtheorem{example}[definition]{Example}

\theoremstyle{plain}
\newtheorem{theorem}[definition]{Theorem}
\newtheorem{lemma}[definition]{Lemma}
\newtheorem{corollary}[definition]{Corollary}
\newtheorem{proposition}[definition]{Proposition}

\numberwithin{equation}{section}

\newcommand{\R}{\ensuremath{\mathbb{R}}}     
\newcommand{\N}{\ensuremath{\mathbb{N}}}     
\newcommand{\Z}{\ensuremath{\mathbb{Z}}}     
\newcommand{\Q}{\ensuremath{\mathbb{Q}}}     
\newcommand{\eps}{\varepsilon}

\begin{document}
\author[J.M.~Thuswaldner]{J\"org M.~Thuswaldner}
\address{Chair of Mathematics and Statistics, University of Leoben, Franz-Josef-Strasse 18, A-8700 Leoben, Austria}
\email{joerg.thuswaldner@unileoben.ac.at}
\title[$\boldsymbol{\beta}$-adic Halton sequences]{Discrepancy bounds for $\boldsymbol{\beta}$-adic Halton sequences}
\dedicatory{Dedicated to Professor Robert F. Tichy on the occasion of his 60$^{\,th}$ birthday}

\begin{abstract} 
Van der Corput and Halton sequences are well-known low-dis\-cre\-pan\-cy sequences. Almost twenty years ago Ninomiya defined analogues of van der Corput sequences for $\beta$-numeration and proved that they also form low-dis\-cre\-pan\-cy sequences if $\beta$ is a Pisot number. Only very recently Robert Tichy and his co-authors succeeded in proving that $\boldsymbol{\beta}$-adic Halton sequences are equidistributed for certain parameters $\boldsymbol{\beta}=(\beta_1,\ldots,\beta_s)$ using methods from ergodic theory. In the present paper we continue this research and give discrepancy estimates for $\boldsymbol{\beta}$-adic Halton sequences for which the components $\beta_i$ are $m$-bonacci numbers. Our methods are quite different and use dynamical and geometric properties of Rauzy fractals that allow to relate $\boldsymbol{\beta}$-adic Halton sequences to rotations on high dimensional tori. The discrepancies of these rotations can then be estimated by classical methods relying on W.~M.~Schmidt's Subspace Theorem.
\end{abstract}

\date{\today}
\thanks{Supported by projects I1136 and P27050 granted by the Austrian Science Fund (FWF)}

\maketitle
	
\section{Introduction}
Given $q\in\N$ with $q\ge 2$, each integer $n \ge 0$ admits a unique {\em $q$-ary expansion} $n=\sum_{j=0}^L\eps_j(n)q^j$ with $\eps_j(n) \in \{0,\ldots, q-1\}$ and $\eps_L(n)\not=0$ for $L\not=0$. Using this expansion we can define the so-called {\em van der Corput sequence}
\[
v_q(n)= \sum_{j=0}^{L}\eps_j(n)q^{-j-1} \in [0,1)  \qquad(n \ge  0). 
\]
As mentioned for instance in Kuipers and Niederreiter~\cite[Chapter~2, Section~3]{Kuipers-Niederreiter:74} this sequence has optimal equidistribution properties modulo $[0,1)$ and, hence, is a so-called {\em low-discrepancy sequence}. A generalization to higher dimensions $s \in\N$ is provided by the {\em Halton sequence} which is defined for each parameter vector $\boldsymbol{q}=(q_1,\ldots,q_s)$, $q_i\ge 2$, by
\[
h_{\boldsymbol{q}}(n) = (v_{q_1}(n),\ldots, v_{q_s}(n))\qquad(n\ge 0).
\]
Halton sequences admit strong equidistribution properties modulo $[0,1)^s$ for parameter vectors $\boldsymbol{q}$ with pairwise relatively prime entries (see \cite[p.~129]{Kuipers-Niederreiter:74} or \cite{Halton:60}).

These concepts can be carried over to so-called {\em linear recurrent number systems} (see {\it e.g.}~\cite{Grabner-Tichy:90,Petho-Tichy:89}) and {\em $\beta$-expansions} ({\it cf. }~\cite{Akiyama:07,Parry:60,Renyi:57}). As we start out from linear recurrent number systems, we briefly recall the definition of these objects. For $d \in\N$ let 
\begin{equation}\label{eq:generalrecurrence}
G_{k+d}= a_1 G_{k+d-1}+a_2G_{k+d-2}+\cdots+a_dG_k \qquad(k\ge 0)
\end{equation}
be a linear recurrence with integral coefficients $a_1,\ldots, a_d$ and integral initial values $G_0,\ldots,G_{d-1}$. Suppose the coefficients satisfy $a_1\ge a_2\ge\cdots \ge a_d\ge1$ ($a_1 > 1$ is needed for $d=1$) and the conditions
\begin{equation}\label{eq:recurrenceconditions}
G_0=1,\quad G_k = a_1G_{k-1}+\cdots + a_kG_0+1 \quad\hbox{for}\quad k\in \{1,\ldots, d-1\}
\end{equation}
are in force. Then we can expand each $n\in\mathbb{N}$ uniquely by a {\em greedy algorithm} as
\begin{equation}\label{eq:greedy}
n=\sum_{j=0}^\infty \varepsilon_j(n) G_j,
\end{equation}
where the digit string $\ldots\eps_1(n)\eps_0(n)$ satisfies
\begin{equation}\label{eq:lex}
\eps_k(n)\ldots\eps_1(n)\eps_0(n)0^\infty < (a_1\ldots a_{d-1} (a_d-1))^{\infty}
\end{equation}
in lexicographic order for each $k\in\N$ (see~\cite{Petho-Tichy:89} for details).  

Let $\beta$ be the dominant root of the characteristic equation of \eqref{eq:generalrecurrence}. Then the {\em $\beta$-adic van der Corput sequence} is defined by
\begin{equation}\label{eq:betavanderCorput}
V_\beta(n) = \sum_{j\ge 0} \varepsilon_j(n) \beta^{-j-1}.
\end{equation}
One of the problems in the $\beta$-adic case is that the asymmetry of the language of possible digit strings $\ldots\varepsilon_1(n)\varepsilon_0(n)$ entails that reflecting expansions on the decimal point destroys the equidistribution properties (see~\cite[Section~2]{Steiner:09}). However, symmetric languages still lead to equidistribution of $V_\beta(n)$ modulo $1$. For instance, as observed in Barat and Grabner~\cite[Section~4]{Barat-Grabner:96}, $\beta$-adic van der Corput sequences are equidistributed modulo $1$ for recurrences of the types
\begin{equation}\label{eq:Barat}
\begin{array}{rll}
\displaystyle G_{k+d}&\displaystyle =a(G_{k+d-1}+\cdots+G_{k}) &\quad(d\ge 2,\; a\in \N) \qquad\hbox{and}\\[6pt]
\displaystyle G_{k+2}&\displaystyle=(a+1)G_{k+1} + aG_k &\quad (a\in \N). 
\end{array}
\end{equation}
Ninomiya~\cite{Ninomiya:98} came up with a slightly different definition of $\beta$-adic van der Corput sequences that are equidistributed modulo $1$ for every Pisot number $\beta$. In particular, instead of reflecting expansions on the decimal point he reorders them w.r.t.\ the reverse lexicographic order. However, as mentioned in \cite[Acknowledgements]{Ninomiya:98}, if the linear recurrence is of the form \eqref{eq:Barat} then Ninomiya's sequence agrees with the one defined in \eqref{eq:betavanderCorput}. Generalizing Ninomiya's construction, Steiner~\cite{Steiner:09} defines and studies van der Corput sequences for abstract numeration systems in the sense of Lecomte and Rigo~\cite{Lecomte-Rigo:01,Lecomte-Rigo:02}. 

Analogously to the classical Halton sequence we define its $\boldsymbol{\beta}$-adic variant by
\[
H_{\boldsymbol{\beta}}(n) = (V_{\beta_1}(n),\ldots,V_{\beta_s}(n)) \qquad (\boldsymbol{\beta}=(\beta_1,\ldots,\beta_s)).
\]
The first result on equidistribution properties of $\boldsymbol{\beta}$-adic Halton sequences is due to Robert Tichy and his co-authors (see Hofer {\it et al.}~\cite{Hofer-Iaco-Tichy:15}) and reads as follows. Let $(G^{(1)}_k),\ldots, (G^{(s)}_k)$ be linear recurrent sequences  of the form \eqref{eq:Barat} (indeed, they could exhibit a slightly larger class, see \cite[Lemma~1]{Hofer-Iaco-Tichy:15}), and let $\beta_1,\ldots, \beta_s$ be the dominant roots of the associated characteristic equations. Then, under suitable assumptions on the algebraic independence of the elements $\beta_i$, the sequence $H_{\boldsymbol{\beta}}(n) $ is equidistributed modulo $[0,1)^s$ for $\boldsymbol{\beta}=(\beta_1,\ldots, \beta_s)$. This result is proved by methods from ergodic theory and provides no information on the discrepancy (see also the generalizations proved in \cite{JassovaEtAl:15}). Recently, Drmota~\cite{Drmota:15} considered ``hybrid'' Halton sequences of vectors containing classical van der Corput sequences plus one component which is equal to the van der Corput sequence $V_\varphi(n)$, with $\varphi$ being the golden ratio (corresponding to the Fibonacci sequence). His approach is different and he is able to give good bounds on the discrepancy that are close to optimality.

To state some of the mentioned results more precisely we introduce some notation. For $s \in\N$ and $A\subset [0,1)^s$ we denote by $\mathbf{1}_A$ the characteristic function of $A$. For a given sequence $(\mathbf{y}_n)_{n \ge 0}$ in $[0,1)^s$ we define the {\em (star) discrepancy} by
\[
D_N((\mathbf{y}_n)_{n\ge 0}) = \sup_{0< \omega_1,\ldots, \omega_s \le 1} \left\vert
\frac{1}{N}\sum_{n=0}^{N-1}\mathbf{1}_{[0,\omega_1)\times\cdots\times [0,\omega_s)}(\mathbf{y}_n)-\omega_1\cdots\omega_s
\right\vert.
\]
It is well known that $D_N((\mathbf{y}_n)_{n\ge 0})$ tends to $0$ for $N\to \infty$ if and only if $(\mathbf{y}_n)_{n\ge 0}$ is equidistributed modulo $[0,1)^s$. Moreover, this quantity provides a gauge for the quality of the equidistribution of a sequence. For more on discrepancy we refer to \cite{Drmota-Tichy:97,Kuipers-Niederreiter:74}.

Using this notion the above-mentioned results read as follows: for the van der Corput sequence it is well known that $D_N((v_q(n))_{n\ge 0}) \ll \log N/N$ and for the $s$-dimensional Halton sequence we have $D_N((h_{\boldsymbol{q}}(n))_{n\ge 0})\ll (\log N)^s/N$ (see {\it e.g.}~\cite{Kuipers-Niederreiter:74}; in \cite{Hofer:16} it is shown that this result extends to Halton sequences defined in terms of number systems with rational bases in the sense of \cite{AFS:08}). In the $\beta$-adic case Ninomiya~\cite{Ninomiya:98} shows that $D_N((V_\beta(n))_{n\ge 0}) \ll \log N/N$ as well, while for the $\boldsymbol{\beta}$-adic Halton sequences $H_{\boldsymbol{\beta}}(n)$ no such estimates seem to be known. Concerning the hybrid case Drmota~\cite{Drmota:15} shows that for each $\eps>0$ one has $D_N((v_{q_1}(n),\ldots , v_{q_s}(n), V_\varphi(n))_{n\ge 0}) \ll N^{\eps-1}$.
 
A large number of other variants of van der Corput and Halton sequences has been studied in the literature and for many of them strong discrepancy  bounds can be derived (see {\it e.g.}~\cite{Carbone:12,FKP:15,Faure-Lemieux:10,FIN:01,Hofer:16,HKLP:09,Mori-Mori:12}).
 
Our aim is to give estimates for the discrepancy of $\boldsymbol{\beta}$-adic Halton sequences that are defined in terms of $m$-bonacci numbers. To be more precise, for $m\ge 2$ define the {\em $m$-bonacci sequence} $(F_k^{(m)})_{k\ge 0}$ by
\begin{equation}\label{eq:Fmrec}
\begin{array}{rll}
\displaystyle F_k^{(m)} & \displaystyle = 2^k,  \quad\qquad &0\le k \le m-1,\\[5pt]
\displaystyle F_k^{(m)} & \displaystyle = \sum_{j =1}^{m} F_{k-j}^{(m)}, \quad &k \ge m.
\end{array}
\end{equation}
It is easy to see that $(F_k^{(m)})_{k\ge0}$ satisfies the conditions in \eqref{eq:recurrenceconditions} for $m\ge 2$ and, hence, the $m$-bonacci sequence admits a unique ``greedy'' expansion \eqref{eq:greedy} for each $n\in \mathbb{N}$. In what follows, we denote the dominant root of the characteristic equation of the recurrence of $(F_k^{(m)})$ by $\varphi_m$ and call it the {\it $m$-bonacci number}. Choose pairwise distinct integers $m_1,\ldots, m_s$ greater than $1$, and let $\beta_i=\varphi_{m_i}$ for $i\in\{1,\ldots,s\}$ and $\boldsymbol{\beta}=(\beta_1,\ldots,\beta_s)$. Our main result will establish a discrepancy estimate for the sequence $H_{\boldsymbol{\beta}}(n)$ for these choices of vectors $\boldsymbol{\beta}$ (see Theorem~\ref{thm:main}).

\section{Preliminaries}

In \cite[Lemma~6]{Drmota:15} the discrepancy of the van der Corput sequence $V_\varphi(n)$ associated with the golden ratio $\varphi$ was expressed in terms of local discrepancies of certain rotations. To this matter, relations between Ostrowski expansions and irrational rotations as discussed in \cite[Section~3.2]{Haynes:16} are used. As anticipated in \cite[Section~1]{Drmota:15}, this doesn't generalize to higher degree recurrences and should be replaced by ``a distribution analysis of linear sequences (modulo 1) in Rauzy fractal type sets''. This is what we do for our class of $m$-bonacci recurrences. Before we can define these objects we need some notation.

Let $m\ge 2$ be fixed. A {\em substitution} $\sigma$ over a finite alphabet $\mathcal{A}=\{1,\ldots, m\}$ is an endomorphism over the free monoid $\mathcal{A}^*$ of finite words over $\mathcal{A}$ (equipped with the operation of concatenation). The {\em incidence matrix} $B_\sigma$ of $\sigma$ is the $m\times m$ matrix given by $B_\sigma=(|\sigma(j)|_i)_{i,j\in\mathcal{A}}$. Here, for a word $w\in \mathcal{A}^*$ we denote the number of occurrences of the letter $i\in \mathcal{A}$ in $w$ by $|w|_i$. Moreover, we will use $|w|$ for the number of letters of $w$. The {\it abelianization map} $\mathbf{l}:\mathcal{A}^* \to \N^m$ is defined by $\mathbf{l}(w)=(|w|_1,\ldots,|w|_m)$. It satisfies the relation $\mathbf{l}\sigma(w)=B_\sigma\mathbf{l}(w)$ for each $w\in \mathcal{A}^*$. There is a natural extension of a substitution $\sigma$ to the space of infinite words $\mathcal{A}^{\N}$, which is equipped with the product topology of the discrete topology on $\mathcal{A}$. If we assume that $\sigma(1)$ starts with the letter $1$ (what we will do throughout this paper) it is easy to see that the limit
\begin{equation}\label{eq:fix}
u=u_1u_2\ldots=\lim_{n\to\infty} \sigma^n(11\ldots) 
\end{equation}
exists and satisfies $\sigma(u)=u$. 

A substitution $\sigma$ is called a {\em unit Pisot substitution} if the characteristic polynomial of $B_\sigma$ is the minimal polynomial of a Pisot unit. This class of substitutions is well studied (see {\it e.g.}~\cite{ABBLS:15,BST:10,Ito-Rao:06} for their dynamical and geometric properties). We wish to associate a {\em Rauzy fractal} with a unit Pisot substitution $\sigma$. To this matter let $\mathbf{u}$ and $\mathbf{v}$ be the normalized right and left eigenvector of $B_\sigma$ associated with the dominant Pisot eigenvalue of $B_\sigma$, respectively. Then the orthogonal space $\mathbf{v}^\bot$ is the contractive invariant hyperspace of the matrix $B_\sigma$. We define $\pi_c$ to be the projection of $\mathbb{R}^m$ along $\mathbf{u}$ to this contracting hyperspace. If $u=u_1u_2\ldots$ is given by \eqref{eq:fix}, the {\em Rauzy fractal} $\mathcal{R}$ associated with the substitution $\sigma$ is defined by 
\[
\mathcal{R} = \overline{\{\pi_c\mathbf{l}(u_1u_2\ldots u_n) \;:\; n\in \mathbb{N}\}}.
\] 
Obviously we have that $\mathcal{R}$ is a subset of $\mathbf{v}^\bot$ which can be viewed as the union of the {\it subtiles}
\[
\mathcal{R}(i) = \overline{\{\pi_c\mathbf{l}(u_1u_2\ldots u_n) \;:\; u_{n+1}=i,\, n\in \mathbb{N}\}}\qquad(i\in\mathcal{A}).
\] 
These sets first occur in \cite{Rauzy:82} and have been extensively studied since then ({\it cf. e.g.}~\cite{ABBLS:15,Arnoux-Ito:01,BST:10,Ito-Rao:06}). The set $\mathcal{R}(i)$, $i\in\mathcal{A}$, is a compact subset of $\mathbf{v}^{\bot}$ which is the closure of its interior and whose boundary has zero measure (see for instance~\cite{BST:10}). 

Under certain conditions (such as for example the {\it super coincidence condition} introduced in \cite{Ito-Rao:06}) the subtiles $\mathcal{R}(i)$ are essentially disjoint and $\mathcal{R}$ forms a fundamental domain of the lattice $\mathcal{L}=\langle \pi_c(\mathbf{e}_1-\mathbf{e}_2),\ldots, \pi_c(\mathbf{e}_1-\mathbf{e}_m) \rangle_\mathbb{Z}$ and, hence, tiles $\mathbf{v}^\bot$ by translates taken from this lattice (see \cite[Theorems~1.2 and ~1.3]{Ito-Rao:06}). Here $\mathbf{e}_i$ ($1\le i \le m$) are the standard basis vectors of $\mathbb{R}^m$. 

To state a set equation for the subtiles $\mathcal{R}(i)$ we need the following  {\em prefix-suffix graph} $G_\sigma$ ({\it cf}.~\cite[Section~3]{Canterini-Siegel:01}). The set of vertices of this labeled directed graph is given by the set of letters $i\in \mathcal{A}$ and there is an edge from $i$ to $j$ labeled by $(p,i,s)\in\mathcal{A}^*\times \mathcal{A}\times \mathcal{A}^*$ if $\sigma(j)=pis$. This edge will be written as $i \xrightarrow{(p,i,s)}j$ or just as $i \xrightarrow{p}j$, since $(p,i,s)$ is completely determined by $j$ and $p$.

Using this graph we can view the subtiles $\mathcal{R}(i)$ as the unique solution of the {\it graph directed iterated function system} (in the sense of Mauldin and Williams~\cite{Mauldin-Williams:88})
\begin{equation}\label{eq:setequation}
\mathcal{R}(i) = \bigcup_{i\xrightarrow{(p,i,s)}j} B_\sigma R(j) + \pi_c\mathbf{l}(p)\qquad (i\in \mathcal{A}).
\end{equation}
Here the union, which is extended over all edges in $G_\sigma$ starting at the vertex $i$ is always essentially disjoint.
This set equation goes back to Sirvent and Wang~\cite{Sirvent-Wang:02}. 

We now define the class of substitutions we will need for our purposes. Let $m\ge 2$ be fixed. To the linear recurrent sequence $(F_k^{(m)})$ given by \eqref{eq:Fmrec} we associate the substitution
\[
\sigma_m(i) = \begin{cases}
1(i+1) & \hbox{for } i < m, \\
1 &\hbox{for } i=m,
\end{cases} 
\]
and denote its incidence matrix by $B_m=B_{\sigma_m}$. By induction we see that
\begin{equation}\label{eq:len}
F_{k}^{(m)} = |\sigma_m^k(1)| \qquad(k\ge 0).
\end{equation}
It is well known that $\sigma_m$ is a unit Pisot substitution for each $m\ge 2$ (see {\it e.g.}~\cite[Section~2.3]{Berthe-Siegel:05}). Thus we may associate a Rauzy fractal  $\mathcal{R}_m=\bigcup_{i\in \mathcal{A}}\mathcal{R}_m(i)$ with it. It has the following properties.

\begin{lemma}\label{lem:over}
Let $m\ge 2$. For each $k\in\N$ define the collection
\[
\mathcal{S}^{(m)}_k = \mathcal{S}_k  =
\left\{
B_\sigma^{k}\mathcal{R}_m(i_k) + \pi_c\mathbf{l}\Big(\sigma_m^{k-1}(p_{k-1})\ldots\sigma_m(p_1)p_0\Big)   \;:\; i_0 \xrightarrow{p_0} \cdots  \xrightarrow{p_{k-1}}i_k
\in G_{\sigma_m} \right\},
\]
where the range reaches over all walks of length $k$ in $G_{\sigma_m}$. 
Then the elements of $\mathcal{S}_k$ overlap only on their boundaries,
\begin{equation}\label{eq:subdiv}
\mathcal{R}_m = \bigcup_{S\in \mathcal{S}_k} S,
\end{equation}
and $\mathcal{R}_m$ forms a fundamental domain of the lattice $\mathcal{L}=\mathcal{L}_m$.
\end{lemma}

The elements of $\mathcal{S}_k$ are called the {\em level $k$ subtiles} of $\mathcal{R}_m$.

\begin{proof}
It follows from \cite[Example~3.1 and Theorem~1.2]{Ito-Rao:06} that the subtiles $\mathcal{R}_m(i)$, $i\in\mathcal{A}$, are essentially disjoint and that their union $\mathcal{R}_m$ is a fundamental domain of the lattice $\mathcal{L}=\mathcal{L}_m$ in $\mathbf{v}^\bot=\mathbf{v}_m^\bot$ (see also~\cite{Akiyama:02}). Using the disjointness of the subtiles and iterating \eqref{eq:setequation} for $k$ times we get \eqref{eq:subdiv}, where the union has no essential overlaps.
\end{proof}

The prefix-suffix graph $G_{\sigma_m}$ of $\sigma_m$ is easy to calculate and a drawing of it is provided in Figure~\ref{fig:prefixsuffix}.

\begin{figure}[h] 
\hskip 0.8cm \xymatrix{
&&&&&&\\
&&&&&&\\
*++[o][F]{1}
\ar@/^4ex/[rr]^(0.7){(\epsilon,1,4)} 
\ar@/^10ex/[rrrr]^(0.7){(\epsilon,1,m)}
\ar@/^16ex/[rrrrr]^(0.7){(\epsilon,1,\epsilon)}
\ar@(lu,u)[]^{(\epsilon,1,2)}  \ar[r]_(0.5){(\epsilon,1,3)} &*++[o][F]{2} \ar@/^7ex/[l]^{(1,2,\epsilon)}&*++[o][F]{3}\ar@/^7ex/[l]^{(1,3,\epsilon)}&....&*++[o][F]{m-1} \ar@/^7ex/[l]^{(1,m-1,\epsilon)}&*++[o][F]{m} \ar@/^7ex/[l]^{(1,m,\epsilon)}
}

\caption{The prefix-suffix graph $G_{\sigma_m}$ of $\sigma_m$ ($\epsilon$ denotes the empty word). \label{fig:prefixsuffix}}
\end{figure}
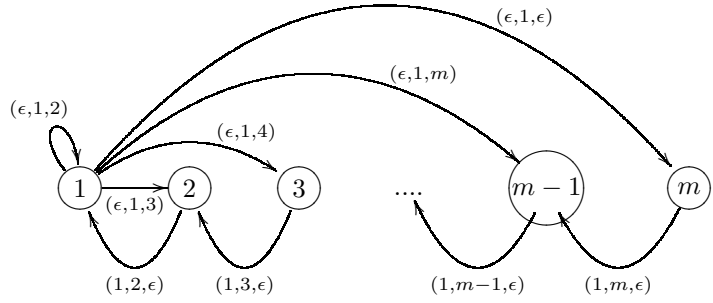

We equip $\mathbf{v}_m^\bot$ with the Lebesgue measure $\lambda_{\mathbf{v}_m}$ that is normalized in a way that a fundamental mesh of $\mathcal{L}_m$ has measure $1$. Using \cite[Equation~(3.3)]{Ito-Rao:06} it is easy to see that 
\begin{equation}\label{eq:Rmeasure}
\lambda_{\mathbf{v}}(\mathcal{R}_m(i)) = \varphi_m^{-i} \qquad (1\le i \le m).
\end{equation}

In what follows, whenever convenient we will write $\mathcal{L}=\mathcal{L}_m$, $\mathbf{v}=\mathbf{v}_m$, $\pi_c=\pi_{c,m}$, {\em etc.}\ to emphasize for which $m$ we use the object in question.

\section{$\beta$-adic van der Corput sequences and torus rotations}

The  main result of this section relates the $\varphi_m$-adic van der Corput sequence to a rotation on an $(m-1)$-dimensional torus. In its statement and in the sequel we will write $|A|$ for the Lebesgue measure of a measurable set $A\subset \R$.

\begin{proposition}\label{prop:makeRotation}
Fix $m\ge 2$, let $(F_j^{(m)})_{j \ge 0}$ be the $m$-bonacci sequence with dominant root $\varphi_m$, and choose $k\in\mathbb{N}$ arbitrary. For $n\in\N$ let 
\begin{equation}\label{eq:mborep}
n=\sum_{j\ge 0} \varepsilon_j(n) F_j^{(m)}
\end{equation}
be the expansion of $n$ in the linear recurrent number system associated with $(F_j^{(m)})$. Let $r\in\{0,\ldots,m-1\}$ be given in a way that $\varepsilon_{k-r-1}(n)=0$, $\varepsilon_{k-r}(n)=\cdots=\varepsilon_{k-1}(n)=1$. 
Then the following assertions hold.
\begin{enumerate}
\item  \label{it:Os1} For 
\[
\mu_k=\sum_{j=0}^{k-1} \varepsilon_{k-1-j}(n)\varphi_m^{j}
\]
we have that
\[
V_{\varphi_m}(n) \in
\left[ 
\frac{\mu_k}{\varphi_m^k},\frac{\mu_k+\varphi_m^r-\sum_{i=0}^{r-1}\varphi_m^{i}}{\varphi_m^k}
\right).
\]

\item \label{it:Os2} For 
\[
\nu_k = \sum_{j=0}^{k-1} \varepsilon_j(n)F_j^{(m)} 
\]
we have that
\[
n\pi_c({\mathbf{e}_1}) \in \nu_k\pi_c({\mathbf{e}_1}) + B_m^k \bigcup_{i=1}^{m-r} \mathcal{R}_m(i) + \mathcal{L}_m.
\]
\end{enumerate}
The measures of the occurring sets agree, {\em i.e.}, we have
\begin{equation}\label{eq:measureequality}
\left| 
\left[ 
\frac{\mu_k}{\varphi_m^k},\frac{\mu_k+\varphi_m^r-\sum_{i=0}^{r-1}\varphi_m^{i}}{\varphi_m^k}
\right)
\right|
=
\lambda_{\mathbf{v}}\left(
\nu_k\pi_c({\mathbf{e}_1}) + B_m^k \bigcup_{i=1}^{m-r} \mathcal{R}_m(i)
\right).
\end{equation}
\end{proposition}

\begin{proof}
To show item (\ref{it:Os1}) just note that according to \eqref{eq:lex} (see also \cite[Theorem~1]{Petho-Tichy:89}) the language of possible digit strings  $\ldots \varepsilon_2(n)\varepsilon_1(n)\varepsilon_0(n)$ in \eqref{eq:mborep} is given by the set of finite words in $\{0,1\}^{*}$ in which blocks of $m$ consecutive occurrences of $1$ are forbidden (here we cut leading zeros to get finite sequences). The assertion then follows by easy calculations using the fact that $\varphi_m^m=\varphi_m^{m-1}+\cdots+\varphi_m+1$.

To prove item (\ref{it:Os2}) first consider the labels of a walk in the prefix-suffix graph $G_{\sigma_m}$. According to Figure~\ref{fig:prefixsuffix} we see that the prefixes of these labels are either $\epsilon$ (the empty word) or $1$. Moreover, each sequence of prefixes can occur in the labeling of a walk subject to the condition that it contains no block of $m$ prefixes $1$. Looking at the lengths of the prefixes, this reflects the restriction on the language of digit strings for the linear recurrent sequence $(F_j^{(m)})$. More precisely, if $n$ is given as in \eqref{eq:mborep} then there is a (unique) walk 
$
w:i_0\xrightarrow{p_0}i_1\xrightarrow{p_1}\cdots 
$
in the prefix-suffix graph such that $p_j=1^{\eps_j(n)}$ ($j\in\N$) and vice versa. 
For this walk $w$, using \eqref{eq:len} we get
\begin{equation}\label{eq:DTn}
\begin{array}{rl}
\displaystyle
n&\displaystyle = \sum_{j\ge 0} \varepsilon_j(n) F_j^{(m)} = \sum_{j\ge 0} \varepsilon_j(n) |\sigma_m^j(1)| = \sum_{j\ge 0} |\sigma_m^j(1^{\varepsilon_j(n)})|=\sum_{j\ge 0} |\sigma_m^j({p_j})| \\[15pt]
&\displaystyle = |    \ldots  \sigma_m^2(p_2)\sigma_m(p_1)p_0 |
\end{array}
\end{equation}
(note that since the digits $\eps_j(n)$ are eventually $0$, the prefixes $p_j$ are eventually empty so that we can truncate the sums to finite sums in this formula).
Thus $\ldots  \sigma_m^2(p_2)\sigma_m(p_1)p_0$ is a word of length $n$ and since by the definition of $\mathcal{L}_m$ we have 
\begin{equation*}
\pi_c \mathbf{l}(v_1\ldots v_n) = \pi_c \sum_{i=1}^n{\mathbf e}_{v_i} \equiv n \pi_c(\mathbf{e}_1) \pmod{\mathcal{L}_m}
\end{equation*} 
for each word $v_1\ldots v_n \in \mathcal{A}^n$ we gain
\begin{equation}\label{eq:nequiv}
n\pi_c(\mathbf{e}_1) \equiv \pi_c\mathbf{l}\Big(\ldots  \sigma_m^2(p_2)\sigma_m(p_1)p_0\Big) \pmod{\mathcal{L}_m}.
\end{equation}
In the same way we derive that
\begin{equation}\label{eq:nequiv2}
\nu_k\pi_c(\mathbf{e}_1) \equiv \pi_c\mathbf{l} \Big(\sigma_m^{k-1}(p_{k-1}) \ldots \sigma_m(p_1)p_0\Big) \pmod{\mathcal{L}_m}.
\end{equation}

We need to identify the level $k$ subtile of $\mathcal{R}_m$ containing a representative of $n\pi_c(\mathbf{e}_1)$ modulo $\mathcal{L}_m$.
Since $B_m$ acts as a uniform contraction on $\mathbf{v}^\bot$ we have $B_m^{\ell}\mathcal{R}_m(i) \to \{\mathbf{0}\}$ for $\ell\to\infty$ in Hausdorff metric. Thus, iterating the set equation \eqref{eq:setequation} {\it ad infinitum} for each $j_0\in \mathcal{A}$ and taking the union over $j_0\in\mathcal{A}$ we get
\begin{equation}\label{eq:adinfinitum}
\mathcal{R}_m =  \bigcup_{j_0\xrightarrow{q_0} j_1\xrightarrow{q_{1}}\cdots}  \pi_c\mathbf{l} \Big(\ldots \sigma_m^{2}(q_{2}) \sigma_m(q_1)q_0\Big),
\end{equation}
where the union runs over all infinite walks of $G_{\sigma_m}$. The walk $w$ corresponding to the expansion \eqref{eq:DTn} of $n$ starts with $w_k:i_0\xrightarrow{p_0}i_1\xrightarrow{p_1}\cdots\xrightarrow{p_{k-1}}i_k$ and occurs in the range of the union in \eqref{eq:adinfinitum}. Thus we restrict this union to the infinite walks starting with $w_k$. Doing this we produce the subtile of level $k$ corresponding to the first $k$ digits of $n$. In particular, we get again from the set equation \eqref{eq:setequation} that
\begin{equation}\label{eq:subn}
\begin{array}{l}
\displaystyle B_m^{k}\mathcal{R}_m(i_k) + \pi_c\mathbf{l} \Big(\sigma_m^{k-1}(p_{k-1}) \ldots \sigma_m(p_1)p_0\Big) = \\[8pt]
 \displaystyle \hskip 1.8cm \bigcup_{i_k=j_k\xrightarrow{q_k} j_{k+1}\xrightarrow{q_{k+1}}\cdots}  \pi_c\mathbf{l} \Big(\ldots \sigma_m^{k+1}(q_{k+1})\sigma_m^{k}(q_{k})\sigma_m^{k-1}(p_{k-1}) \ldots \sigma_m(p_1)p_0\Big),
\end{array}
\end{equation}
and the element $\pi_c\mathbf{l}\big(\ldots  \sigma_m^2(p_2)\sigma_m(p_1)p_0\big)$ is contained in this union for some $i_k\in\mathcal{A}$.
We inspect the prefix-suffix graph $G_{\sigma_m}$ in Figure~\ref{fig:prefixsuffix} to specify the possible values of $i_k$. Since $p_{k-r-1}=\eps_{k-r-1}(n)=0$ and $p_{k-r}=\eps_{k-r}(n)=\cdots = p_{k-1}=\eps_{k-1}(n)=1$ we see that $i_k\in\{1,\ldots, m-r\}$, as the other vertices cannot be accessed by a walk ending with a block of prefixes $1$ of length $r$ in $G_{\sigma_m}$. 
Thus
\begin{equation}\label{eq:fastfertig}
\pi_c\mathbf{l}\Big(\ldots  \sigma_m^2(p_2)\sigma_m(p_1)p_0\Big) \in B_m^{k}\mathcal{R}_m(i_k) + \pi_c\mathbf{l} \Big(\sigma_m^{k-1}(p_{k-1}) \ldots \sigma_m(p_1)p_0\Big)
\end{equation}
holds for some $i_k\in\{1,\ldots, m-r\}$. The result in item\eqref{it:Os2} now follows by inserting \eqref{eq:nequiv} and \eqref{eq:nequiv2} in \eqref{eq:fastfertig}.
%

To show \eqref{eq:measureequality} we use \eqref{eq:Rmeasure} together with the fact that $\lambda_\mathbf{v}(B_m A) = \varphi_m^{-1}\lambda_\mathbf{v}(A)$ holds for each measurable $A\subset \mathbf{v}^{\bot}$.
\end{proof}

\begin{remark}\label{rem:DTnum}
The representation of $n$ in terms of lengths of prefixes under powers of $\sigma_m$ in \eqref{eq:DTn} is called the {\em Dumont-Thomas expansion} of $n$. Dumont-Thomas expansions are introduced in~\cite{Dumont-Thomas:89,Dumont-Thomas:93}. As mentioned in Rigo and Steiner~\cite[p.~284]{Rigo-Steiner:05}, an abstract numeration system in the sense of \cite{Lecomte-Rigo:01,Lecomte-Rigo:02} corresponding to  a primitive automaton leads to Dumont-Thomas numeration for a primitive substitution.
\end{remark}

We will need the following tiling result of the intervals  in Proposition~\ref{prop:makeRotation}~(\ref{it:Os1}). 

\begin{lemma}\label{lem:inttes}
Fix $m\ge 2$ and $k \in \N$. To each $n<F_k^{(m)}$ associate $r=r(n)$ and $\mu_k=\mu_k(n)=\sum_{j=0}^{k-1} \varepsilon_{k-1-j}(n)\varphi_m^{j}$ as in Proposition~\ref{prop:makeRotation}. Then the collection of intervals
\[
\mathcal{C}_k^{(m)}=\mathcal{C}_k=\left\{
\left[ 
\frac{\mu_k(n)}{\varphi_m^k},\frac{\mu_k(n)+\varphi_m^{r(n)}-\sum_{i=0}^{r(n)-1}\varphi_m^{i}}{\varphi_m^k}
\right)\;:\; 0\le n < F_k^{(m)}
\right\}
\]
forms a partition of the interval $[0,1)$.
\end{lemma}

\begin{proof}
This follows easily by direct calculation using  $\varphi_m^m=\varphi_m^{m-1}+\cdots+\varphi_m+1$.
\end{proof}

Since by Lemma~\ref{lem:over} the subtiles of level $k$ of $\mathcal{R}_m$ overlap on the boundary, we will need a slightly shifted rotation that always hits the interior of the subtiles.

\begin{corollary}\label{cor:shifted}
Fix $m\ge 2$, let $(F_j^{(m)})_{j \ge 0}$ be the $m$-bonacci sequence with dominant root $\varphi_m$, and choose $k,N\in\mathbb{N}$ arbitrary. For $0\le n < N$ let 
\begin{equation*}
n=\sum_{j\ge 0} \varepsilon_j(n) F_j^{(m)}
\end{equation*}
be the expansion of $n$ in the linear recurrent number system associated with $(F_j^{(m)})$. Let $r\in\{0,\ldots,m-1\}$ be given in a way that $\varepsilon_{k-r-1}(n)=0$, $\varepsilon_{k-r}(n)=\cdots=\varepsilon_{k-1}(n)=1$. 
Then there is a constant $\alpha=\alpha(k,N)\in \mathbf{v}^{\bot}$ depending only on $k$ and $N$ such that
\[
n\pi_c({\mathbf{e}_1}) + \alpha \;\in\; \nu_k\pi_c({\mathbf{e}_1}) +  \bigcup_{i=1}^{m-r} {\rm int}\left( B_m^k\mathcal{R}_m(i)\right) + \mathcal{L}_m \quad (0\le n < N)
\]
with $\nu_k$ as in Proposition~\ref{prop:makeRotation}.
\end{corollary}

\begin{proof}
Let $L$ be defined in a way that $F_{L-1}^{(m)}\le N - 1 < F_{L}^{(m)}$ and $M= \max(k,L)$. Then, arguing as in \eqref{eq:DTn}, the expansion of each $n< N$ can be written as
\[
n=\sum_{j= 0}^{M-1} \varepsilon_j(n) F_j^{(m)} = \sum_{j= 0}^{M-1} |\sigma_m^j(p_j)|
\] 
for some walk labelled by $p_0,\ldots, p_{M-1}$ in the prefix-suffix graph $G_{\sigma_m}$. 
From Proposition~\ref{prop:makeRotation} we know that  
\[
n \pi_c(\mathbf{e}_1)\equiv \pi_c\mathbf{l}\Big(\sigma_m^{M-1}(p_{M-1})\ldots \sigma_m(p_1)p_0\Big)  \;\in\; \nu_k\pi_c{\mathbf{e}_1} + B_m^k \bigcup_{i=1}^{m-r} \mathcal{R}_m(i) + \mathcal{L}_m.
\]
Since $r\in\{0,\ldots,m-1\}$, $i=1$ is always in the range of the union on the right hand side,  by the set equation \eqref{eq:setequation} this right hand side contains all elements
$
\pi_c\mathbf{l}\big(\ldots \sigma_m^{M-1}(p_{M-1})\ldots  \sigma_m(p_1)p_0\big)
$
corresponding to a walk $w$ in the prefix-suffix automaton starting with 
$
i_0 \xrightarrow{p_0} \cdots  \xrightarrow{p_{M-2}}i_{M-1}\xrightarrow{p_{M-1}}1.
$
However, by \eqref{eq:nequiv} and \eqref{eq:subn} this is equivalent to 
\[
n \pi_c(\mathbf{e}_1) + B_m^{M}\mathcal{R}_m(1) \;\subset\; \nu_k\pi_c{\mathbf{e}_1} + B_m^k \bigcup_{i=1}^{m-r} \mathcal{R}_m(i) + \mathcal{L}_m \qquad(0\le n < N).
\]
Since $B_m^{M}\mathcal{R}_m(1)$ has positive measure, while the boundaries of $\mathcal{R}_m(i)$, $i\in \mathcal{A}$, have zero measure we can choose $\alpha\in  B_m^{M}\mathcal{R}_m(1) $ in a way that $n \pi_c(\mathbf{e}_1) + \alpha$ avoids the boundaries of the subtiles for every $n\in \{0,\ldots,N-1\}$ and the result follows.
\end{proof}

\section{Discrepancy estimates}	
	
Using the results of the previous section we will express the discrepancy of $\boldsymbol{\beta}$-adic Halton sequences in terms of local discrepancies of rotations. We first state a generalization of \cite[Lemma~6]{Drmota:15} on the discrepancy of $\varphi_m$-adic van der Corput sequences that is of interest in its own right.

\begin{lemma}\label{lem:local1}
Fix $m\ge 2$, let $N\in \N$ be given, and choose $L$ in a way that $F_{L-1}^{(m)}\le N-1 < F_{L}^{(m)}$. Then
\[
D_N( (V_{\varphi_m}(n))_{n\ge 0} ) \ll \frac{1}{\varphi_m^L} +  \sum_{1\le k \le L} \delta_k,
\]
where
\[
\delta_k = \sup_{S \in \mathcal{S}_k} \left|\frac1N
\sum_{n=0}^{N-1}\mathbf{1}_{S}\Big(n\pi_c(\mathbf{e}_1)+\alpha_k \bmod \mathcal{L}_m\Big) - \lambda_{\mathbf{v}}(S)
\right|
\]
with $\alpha_k=\alpha_k(N)$ chosen as in Corollary~\ref{cor:shifted}.
\end{lemma}

\begin{proof}
Choose $x\in [0,1)$ arbitrary. We can construct a family $I_{1},\ldots,I_{L}, I^*$ of intervals such that $I_k=\emptyset$ or $I_{k} \in \mathcal{C}_{k}$ and $|I^*| < \varphi_m^{-L}$ in a way that 
\begin{equation}\label{eq:Iunion}
I_{1}\cup\cdots\cup I_{L}\cup I^*=[0,x) \quad\hbox{with no overlaps}.
\end{equation}
Indeed, as $\mathcal{C}_{L}$ is a partition of $[0,1)$ by Lemma~\ref{lem:inttes} we can choose $\eps_{L-1}\ldots\eps_0\in\{0,1\}^L$ containing no block of $m$ consecutive $1$s in a way that
$
\mu_L = \sum_{j=0}^{L-1} \eps_{L-1-j}\varphi^{j}
$
satisfies $\mu_L\varphi_m^{-L} \le x < (\mu_L+1)\varphi_m^{-L}$. Then let $I^*=[\mu_L\varphi_m^{-L},x)$. For the definition of $I_k$ let $\xi_k=0\cdot F_{k-1}^{(m)}+\eps_{k-2} F_{k-2}^{(m)}+\cdots + \eps_0F_0^{(m)}$ and use the notation of Lemma~\ref{lem:inttes}. Then, since the $(k-1)$-st digit of $\xi_k$ is $0$ we have $r(\xi_k)=0$ and we set 
\[
I_{k}=
\begin{cases}
\left[ 
\frac{\mu_k(\xi_k)}{\varphi_m^k},\frac{\mu_k(\xi_k)+\varphi_m^{r(\xi_k)}-\sum_{i=0}^{r(\xi_k)-1}\varphi_m^{i}}{\varphi_m^k}
\right)= \left[ \frac{\mu_k(\xi_k)}{\varphi_m^k},\frac{\mu_k(\xi_k)+1}{\varphi_m^k}
\right),& \eps_{k-1} =1,\\
\emptyset,& \eps_{k-1} =0.
\end{cases}
\]
Clearly, we have $I_k\in \mathcal{C}_k$ whenever it is nonempty and with these choices it is easy to check that \eqref{eq:Iunion} is true.

Set $y_n=V_{\varphi_m}(n)$ for convenience. In order to estimate $D_N( (y_n)_{n\ge 0} )$ it is sufficient to deal with intervals of the form $I_k$ and $I^*$ since by the triangle inequality we have
\begin{equation}\label{eq:1sttriangle}
\left\vert \frac{1}{N}\sum_{n=0}^{N-1}\mathbf{1}_{[0,x)}(y_n)-x
\right\vert
\le \sum_{k=1}^{L} \left\vert \frac{1}{N}\sum_{n=0}^{N-1}\mathbf{1}_{I_k}(y_n)-|I_k|
\right\vert + \left\vert \frac{1}{N}\sum_{n=0}^{N-1}\mathbf{1}_{I^*}(y_n)-|I^*|
\right\vert.
\end{equation}
On the right hand side it obviously suffices to take the sum over all values of $k$ corresponding to nonempty sets $I_k$ because we get no contribution otherwise. Since in this case $I_k \in \mathcal{C}_k$ we can apply Proposition~\ref{prop:makeRotation} and Corollary~\ref{cor:shifted}. Indeed, since $\mathcal{C}_{k}$ is a partition of $[0,1)$ by Lemma~\ref{lem:inttes}, Proposition~\ref{prop:makeRotation} yields
\[
y_n \in I_k  \qquad \Longrightarrow \qquad n\pi_c(\mathbf{e}_1) \;\in\; \xi_k\pi_c({\mathbf{e}_1}) + B_m^k \bigcup_{i=1}^{m} \mathcal{R}_m(i) + \mathcal{L}_m.
\]
The right hand side is a union of $m$ level $k$ subtiles and since the elements of $\mathcal{S}_k$ overlap on their boundaries by Lemma~\ref{lem:over}, the converse implication is not true. Thus we have to use Corollary~\ref{cor:shifted} which implies the existence of $\alpha_k \in \mathbf{v}^{\bot}$ such that 
\begin{equation}\label{eq:reversable}
y_n \in I_k  \qquad \Longrightarrow \qquad n\pi_c(\mathbf{e}_1)+\alpha_k \;\in\; \xi_k\pi_c({\mathbf{e}_1}) +  \bigcup_{i=1}^{m} \mathrm{int}(B_m^k\mathcal{R}_m(i)) + \mathcal{L}_m.
\end{equation}
Since the interiors of the elements of $\mathcal{S}_k$ are pairwise disjoint by Lemma~\ref{lem:over}, the implication in \eqref{eq:reversable} can be reversed and it becomes an equivalence. By this equivalence, for each $I_k$ there is a union of $m$ elements $S_1,\ldots, S_{m} \in \mathcal{S}_k$ ({\it viz.}\ $S_i=B_m^k\mathcal{R}_m(i)+\xi_k\pi_c({\mathbf{e}_1}) $) such that
\[
\sum_{n=0}^{N-1}\mathbf{1}_{I_k}(y_n) = \sum_{n=0}^{N-1}\mathbf{1}_{S_1\cup\cdots\cup S_{m}}\Big(n\pi_c(\mathbf{e}_1)+\alpha_k \mod \mathcal{L}_m\Big).
\]
Thus, observing that \eqref{eq:measureequality} yields $|I_k|=\lambda_{\mathbf{v}}(S_1\cup\cdots\cup S_{m})$ and applying the triangle inequality again we gain for each $k\in\{1,\ldots, L\}$
\begin{equation}\label{eq:2ndtriangle}
\left\vert \frac{1}{N}\sum_{n=0}^{N-1}\mathbf{1}_{I_k}(y_n)-|I_k|
\right\vert 
\le 
\sum_{i=1}^{m} \left\vert \frac{1}{N}\sum_{n=0}^{N-1}\mathbf{1}_{S_i}\Big(n\pi_c(\mathbf{e}_1)+\alpha_{k} \mod \mathcal{L}_m\Big)-\lambda_{\mathbf{v}}(S_i) \right\vert
\end{equation}
($\alpha_k$ depends only on on $k$ and $N$). To treat the interval $I^*$ put $J=[\mu_L\varphi_m^{-L},(\mu_L+1)\varphi_m^{-L})$ and note that $J\supset I^*$ is an element of $\mathcal{C}_L$. We have
\begin{equation}\label{eq:gr1}
\sum_{n=0}^{N-1}\mathbf{1}_{I^*}(y_n) \le \sum_{n=0}^{N-1}\mathbf{1}_{J}(y_n)
\le 
\left|
\sum_{n=0}^{N-1}\mathbf{1}_{J}(y_n)-|J|
\right| + |J|
\end{equation}
and, hence,
\begin{equation}\label{eq:gr2}
\left|
\sum_{n=0}^{N-1}\mathbf{1}_{I^*}(y_n)-|I^*|
\right| \le \left|
\sum_{n=0}^{N-1}\mathbf{1}_{J}(y_n)-|J|
\right| + |J| + |I^*|.
\end{equation}
Since $J\in \mathcal{C}_L$ we may apply Proposition~\ref{prop:makeRotation} and Corollary~\ref{cor:shifted} as above and since $|I^*|< |J|=\varphi_m^{-L}$ we gain
\begin{equation}\label{eq:3rdtriangle}
\left\vert
\sum_{n=0}^{N-1}\mathbf{1}_{I^*}(y_n)-|I^*|
\right\vert \le 
\sum_{i=1}^{m} \left\vert \frac{1}{N}\sum_{n=0}^{N-1}\mathbf{1}_{S_i}(n\pi_c(\mathbf{e}_1)+\alpha_L \mod \mathcal{L}_m)-\lambda_{\mathbf{v}}(S_i)  
\right\vert+ \frac{2}{\varphi_m^{L}}
\end{equation}
for some $S_1,\ldots,S_{m}\in \mathcal{S}_L$.
Now we insert \eqref{eq:2ndtriangle} and \eqref{eq:3rdtriangle} in $\eqref{eq:1sttriangle}$ and take the supremum over all $x\in[0,1)$. This yields the result.
\end{proof}

Lemma~\ref{lem:local1} implies Ninomiya's~\cite{Ninomiya:98} low-discrepancy estimate for the $\varphi_m$-adic van der Corput sequence since the level $k$ subtiles of the Rauzy fractals are {\em bounded remainder sets} for the rotation by $\pi_c(\mathbf{e}_i)$ (see~\cite[Theorem~1]{BST:16}). Moreover, its proof paves the way for the proof of the following proposition.

\begin{proposition}\label{prop:product}
Fix $m_1,\ldots,m_s\ge 2$, let $N\in \N$ be given, and choose $L_1,\ldots, L_s$ in a way that $F_{L_j-1}^{(m_j)}\le N - 1 < F_{L_j}^{(m_j)}$. Then for $\boldsymbol{\beta}=(\beta_1,\ldots,\beta_s)=(\varphi_{m_1},\ldots,\varphi_{m_s})$
\begin{equation}\label{eq:prop42}
D_N( (H_{\boldsymbol{\beta}}(n))_{n\ge 0} ) \ll \sum_{i=1}^s\frac{1}{\beta_i^{L_i}} + \sum_{1\le k_1 \le L_1}\cdots \sum_{1\le k_s \le L_s} \delta_{k_1,\ldots, k_s},
\end{equation}
where
\begin{align*}
&\delta_{k_1,\ldots, k_s} = \sup_{S_1 \in \mathcal{S}_{k_1}^{(m_1)},\ldots,S_s \in \mathcal{S}_{k_s}^{(m_s)}} 
\left|\frac1N
\sum_{n=0}^{N-1}\prod_{i=1}^s\mathbf{1}_{S_i}\Big(n\pi_{c,m_i}(\mathbf{e}_1)+\alpha_{k_i} \bmod \mathcal{L}_{m_i}\Big) - \prod_{i=1}^s\lambda_{\mathbf{v}_{m_i}}(S_i)
\right|.
\end{align*}
Here $\alpha_{k_i}=\alpha_{k_i}(N)$ is chosen as in Corollary~\ref{cor:shifted}.

\end{proposition}

\begin{proof}
We start with an arbitrary box $[0,x_1)\times\cdots\times [0,x_s) \subset [0,1)^s$. Each of the intervals $[0,x_i)$ can be partitioned into subintervals $I_1^{i},\ldots, I_{L_i}^{i},I^{*,i}$ with $I_k^{i} \in \mathcal{C}_{k}^{(m_i)}$ or empty and $|I^{*,i}|\le \beta_i^{-L_i}$ (see the proof of Lemma~\ref{lem:local1}). Using the triangle inequality we obtain, setting $\mathbf{y}_n=H_{\boldsymbol{\beta}}(n)$ and $I_{L+1}^i = I^{*,i}$ for convenience,
\begin{align*}
&\left\vert \frac{1}{N}\sum_{n=0}^{N-1}\mathbf{1}_{[0,x_1)\times\cdots\times [0,x_s)}({\mathbf y}_n)-x_1\cdots x_s
\right\vert\le \sum_{k_1=1}^{L_1+1} \ldots \sum_{k_s=1}^{L_s+1} \left\vert \frac{1}{N}\sum_{n=0}^{N-1}\mathbf{1}_{I_{k_1}^1\times \cdots \times I_{k_s}^s}(\mathbf{y}_n)-|I_{k_1}^{1}|\cdots|I_{k_s}^{s}|
\right\vert.
\end{align*}
For each box $I_{k_1}^{1} \times \cdots \times I_{k_s}^{s}$ with $k_1\le L_1,\ldots, k_s \le L_s$ we can argue in the same way as in the proof of Lemma~\ref{lem:local1} to see that 
\begin{align*}
&\left\vert \frac{1}{N}\sum_{n=0}^{N-1}\mathbf{1}_{I_{k_1}^1\times \cdots \times I_{k_s}^s}(\mathbf{y}_n)-|I_{k_1}^{1}|\cdots|I_{k_s}^{s}|
\right\vert 
\\&
\hskip 1.5cm
\le 
\sum_{j_1=1}^{m_1}\ldots \sum_{j_s=1}^{m_s} \left\vert \frac{1}{N}\sum_{n=0}^{N-1}\prod_{i=1}^s
\mathbf{1}_{I_{k_i}^i}(V_{\beta_i}(n))-
\prod_{i=1}^s\lambda_{\mathbf{v}_{m_i}}(S_{ij_i}) \right\vert
\\&
\hskip 1.5cm
=
 \sum_{j_1=1}^{m_1}\ldots \sum_{j_s=1}^{m_s} \left\vert \frac{1}{N}\sum_{n=0}^{N-1}\prod_{i=1}^s\mathbf{1}_{S_{ij_i}}\Big(n\pi_{c,m_i}(\mathbf{e}_1)+\alpha_{k_{i}} \mod \mathcal{L}_{m_i}\Big)-\prod_{i=1}^s\lambda_{\mathbf{v}_{m_i}}(S_{ij_i}) \right\vert
\end{align*}
for certain $S_{ij} \in \mathcal{S}_{k_i}^{(m_i)}$ ($1\le i \le s,\, 1\le j \le m_i$). If some of the $k_i$ equal $L_i+1$ then blowing up the according intervals $I_{L+1}^i =I^{*,i}$ as in the proof of Lemma~\ref{lem:local1} we can treat them in the same way as the intervals $I_{k_i}^{i}$ with $k_i\le L_i$ causing an overall error bounded by $\sum_{i=1}^s\frac{1}{\beta_i^{L_i}}$. Taking the supremum over all boxes $[0,x_1)\times\cdots\times [0,x_s) \subset [0,1)^s$ gives the result.
\end{proof}

\section{The main result}

We are now in a position to state and prove our main result, a discrepancy estimate for $\boldsymbol{\beta}$-adic Halton sequences with $m$-bonacci arguments.

\begin{theorem}\label{thm:main}
Let $m_1,\ldots,m_s$ be pairwise distinct integers greater than or equal to $2$ and set $\boldsymbol{\beta}=(\varphi_{m_1},\ldots,\varphi_{m_s})$. If $\{1,\varphi_{m_1},\ldots,\varphi_{m_1}^{m_1-1},\ldots, \varphi_{m_s},\ldots,\varphi_{m_s}^{m_s-1}\}$ is linearly independent over $\Q$ then the discrepancy of the $\boldsymbol{\beta}$-adic Halton sequence $H_{\boldsymbol{\beta}}(n)$ satisfies 
\begin{equation}\label{eq:main}
D_N( (H_{\boldsymbol{\beta}}(n))_{n\ge 0} ) \ll N^{\frac{\max\{d_i-(m_i-1)\,:\,1\le i\le s\}}{(m_1-1)+\cdots + (m_s-1)}+\varepsilon}
\end{equation}
for each $\varepsilon > 0$. Here $d_i=\mathrm{dim_B}(\partial \mathcal{R}_{m_i})$, which is strictly smaller than $m_i-1$, denotes the box counting dimension of the boundary of the Rauzy fractal $\mathcal{R}_{m_i}$, $1\le i\le s$. 
\end{theorem}

\begin{remark}\label{rem:independence}
The box counting dimension of $\partial \mathcal{R}_m$ can be calculated explicitly in terms of the so-called {\em boundary graph} (see for instance \cite[Theorem~4.4]{Siegel-Thuswaldner:09} for a formula or \cite[Theorem~3.1]{FFIW:06} for an estimate). 

A sufficient condition for the set  $\{1,\varphi_{m_1},\ldots,\varphi_{m_1}^{m_1-1},\ldots, \varphi_{m_s},\ldots,\varphi_{m_s}^{m_s-1}\}$ to be linearly independent over $\Q$ is that the degree of the extension $\mathbb{Q}(\varphi_{m_1},\ldots,\varphi_{m_s}):\mathbb{Q}$ satisfies $[\mathbb{Q}(\varphi_{m_1},\ldots,\varphi_{m_s}):\mathbb{Q}]=m_1\cdots m_s$.  This condition holds for instance if the integers $m_1,\ldots,m_s$ or the discriminants of the number fields $\mathbb{Q}(\varphi_{m_1}),\ldots,\mathbb{Q}(\varphi_{m_s})$ are pairwise relatively prime (see {\it e.g.} Mordell~\cite{Mordell:53}).
\end{remark}

We set the stage for the proof of Theorem~\ref{thm:main} by establishing a series of preparatory results. First we provide a technical lemma which shows how to conjugate the rotation by the vector $\pi_c(\mathbf{e}_1)$ on the fundamental domain of the lattice $\mathcal{L}_m$ in $\mathbf{v}_m^{\bot}$ to a rotation on the standard torus $\R^{m-1}/\Z^{m-1}\simeq [0,1)^{m-1}$.

\begin{lemma}\label{lem:makecube}
Let $m\ge 2$ be given. Then 
\begin{equation}\label{eq:explicitrotation}
\pi_c({\mathbf e}_1) = 
\sum_{i=2}^{m}\pi_c({\mathbf e}_1-{\mathbf e}_i)\varphi_m^{-i} ,
\end{equation}
{\em i.e.},  the rotation $n\pi_c(\mathbf{e}_1) + \alpha \mod \mathcal{L}_m$, $n\in \mathbb{N}$, is conjugate to the rotation 
\[
n\left(\varphi_m^{-2},\ldots, \varphi_m^{-m}\right)+Q_m\alpha \bmod \Z^{m-1},\qquad n\in \mathbb{N},
\]
by a linear conjugacy $Q_m$.
\end{lemma}

\begin{proof}
The identity in \eqref{eq:explicitrotation} can be verified by a (somewhat tedious) direct computation. Indeed, it is easy to see that each coordinate of \eqref{eq:explicitrotation} is  an element of $\mathbb{Q}(\varphi_m)$. Using this it suffices  to compare the coefficients of $\varphi_m^i$ for each $0\le i\le m-1$ in each coordinate. The conjugacy assertion follows immediately from \eqref{eq:explicitrotation}.
\end{proof}

Choose $m_1,\ldots, m_s \ge 2$. In what follows we denote by 
\begin{equation}\label{eq:Q}
Q=Q_{m_1}\times \cdots \times Q_{m_s} 
\end{equation}
the conjugacy between the rotation 
\begin{equation}\label{eq:rn}
R(n,\boldsymbol{\alpha})=\prod_{i=1}^s (n\pi_{c,m_i}(\mathbf{e}_1)+\alpha_i \bmod \mathcal{L}_{m_i}) 
\end{equation}
with offset $\boldsymbol{\alpha}=(\alpha_1,\ldots, \alpha_s)$ on the $(m_1-1)+\cdots+(m_s-1)$ dimensional torus $\prod_{i=1}^s (\mathbf{v}_{m_i}^\bot / \mathcal{L}_{m_i})$ and the rotation
\begin{equation}\label{eq:qrn}
QR(n,\boldsymbol{\alpha})=\prod_{i=1}^s\left(
n\left(\varphi_{m_i}^{-2},\ldots, \varphi_{m_i}^{-m_i}\right)+Q\alpha_i \bmod \Z^{(m_i-1)} \right)
\end{equation}
on the standard torus $[0,1)^{(m_1-1)+\cdots+(m_s-1)}$. This conjugacy exists by Lemma~\ref{lem:makecube}

In what follows we will need to study properties of products of the form
\[
P=P(k_1,\ldots,k_s)=Q(B_{m_1}^{k_1}\mathcal{R}_{m_1}\times \cdots \times B_{m_s}^{k_s}\mathcal{R}_{m_s}) \in \R^{(m_1-1)+\cdots+(m_s-1)}
\]
with $Q$ as in \eqref{eq:Q}. We will need the {\em box counting dimension} of $\partial P$, whose value doesn't depend on $k_1,\ldots, k_s$. It is defined by (see Falconer~\cite[Section~3.1]{Falconer:90})
\[
D=\mathrm{dim_B}(\partial P)=\lim_{\ell\to\infty}\frac{\log N(\partial P,\ell)}{\ell \log 2},  
\]
where $N(\partial P,\ell)$ is the number of boxes of side length $2^{-\ell}$ in $\R^{(m_1-1)+\cdots+(m_s-1)}$ arranged in a grid centered at $0$ having nonempty intersection with $\partial P$.
The boundary $\partial P$ is given by the union of cartesian products of the shape $(1\le i \le s)$
\[
Q(B_{m_1}^{k_1}\mathcal{R}_{m_1}\times \cdots\times B_{m_{i-1}}^{k_{i-1}}\mathcal{R}_{m_{i-1}}\times\partial B_{m_i}^{k_i}\mathcal{R}_{m_i} \times B_{m_{i+1}}^{k_{i+1}}\mathcal{R}_{m_{i+1}} \times \cdots \times B_{m_s}^{k_s}\mathcal{R}_{m_s})
\]
and, hence, its box counting dimension satisfies ({\it cf.~e.g.}~\cite[Product Formula~7.5]{Falconer:90})
\[
D \le \max\left\{ \sum_{j=1,\,  j\not=i}^{s}  (m_j-1) + \mathrm{dim_B}(\mathcal{R}_{m_i})  \;:\; 1\le i \le s \right\}.
\]
As mentioned above, in \cite{FFIW:06} an explicitly computable bound for $\mathrm{dim_B}(\mathcal{R}_{m})$ is given. Since  $\mathcal{R}_m$ admits a tiling w.r.t.\ the lattice $\mathcal{L}_m$ by Lemma~\ref{lem:over}, it follows from \cite[Theorem~4.1]{Siegel-Thuswaldner:09} that this bound is nontrivial for each $m\ge 2$, and, hence, we have
\begin{equation}\label{eq:boxestimate}
D < (m_1-1)+\cdots+(m_s-1).
\end{equation}

We now turn to a covering property for products of Rauzy fractals.

\begin{lemma}\label{lem:boxdim}
Let $m_1,\ldots,m_s\ge 2$ and $k_1,\ldots,k_s\ge 0$, and set 
\[
P=Q(B_{m_1}^{k_1}\mathcal{R}_{m_1}\times \cdots \times B_{m_s}^{k_s}\mathcal{R}_{m_s}).
\]
Let $D=\mathrm{dim_B}(\partial P)$ be the box counting dimension of $\partial P$ and fix $\varepsilon >0$. Then for each $M \in\N$ we can cover $P$ by boxes $U_1,\ldots,U_r, V_1,\ldots, V_{r'}\subset \R^{(m_1-1)+\cdots+(m_s-1)}$ with $r,r'\ll 2^{M(D+\varepsilon)}$ in the following way. 
\begin{itemize}
\item Each box $U_i$ has empty intersection with the complement of $P$ ($1\le i\le r$).
\item Each  box $V_i$ intersects the boundary of $P$ ($1\le i\le r'$) and 
$$
\sum_{i=1}^{r'}\mu(V_i)\ll 2^{M(D-(m_1-1)-\cdots-(m_s-1)+\varepsilon)},
$$
where $\mu$ denotes the Lebesgue measure on $\R^{(m_1-1)+\cdots+(m_s-1)}$.
\end{itemize}
\end{lemma}

\begin{proof}
Cover $P$ by the collections $\mathcal{K}_\ell$ of boxes of side length $2^{-\ell}$ for each $\ell\in\{1,\ldots, M\}$ arranged in a grid centered at $0$. So for $\ell > 1$ each box in $\mathcal{K}_\ell$ is contained in a unique larger box of $\mathcal{K}_{\ell-1}$ which is called the {\em parent} of this box. Choose the boxes $U_i$ inductively as follows. First take all elements of $\mathcal{K}_1$ that have 
empty intersection with the complement of $P$. For $\ell >1$ take all elements of $\mathcal{K}_\ell$ that have 
empty intersection with the complement of $P$ and which has not been covered so far. Thus $\{U_1,\ldots, U_r\}$ contains, apart from elements of $\mathcal{K}_1$, only elements of $\mathcal{K}_\ell$ whose parents intersect the boundary of $P$. By the definition of the box counting dimension, the number $z_\ell$ of such elements satisfies $z_\ell \ll 2^{\ell(D+\varepsilon)}$. Summing up $z_\ell$ over all $1\le \ell\le M$ we obtain the bound for $r$ asserted in the statement. 

Since the sets $V_i$ can be chosen to be boxes from $\mathcal{K}_M$ in the claimed way by the definition of the box counting dimension the lemma is proved.
\end{proof}

In the last preparatory lemma we recall a classical discrepancy estimate for rationally independent rotations on the torus. We will use the following notation (see {\it e.g.}~Niederreiter~\cite{Niederreiter:73,Niederreiter:10}).  We say that $\boldsymbol{\gamma}\in\R^s$ is of {\em finite type $\eta_0$} if $\eta_0\in\R$ is the infimum of all $\eta\in\R$ for which there exists a positive constant $c=c(\eta,\boldsymbol{\gamma})$ such that
\[
\left(\prod_{j=1}^s\max(|h_j|,1)\right)^{\eta}||(h_1,\ldots,h_s)\cdot\boldsymbol{\gamma}|| \ge c
\]
holds for all $(h_1,\ldots,h_s)\in\Z^s\setminus\{\mathbf{0}\}$.

\begin{lemma}\label{lem:discrepancy}
Let $\boldsymbol{\gamma}=(\gamma_1,\ldots, \gamma_s)\in\R^s$ with algebraic numbers $\gamma_1,\ldots, \gamma_s$ and assume that $\{1,\gamma_1,\ldots, \gamma_s\}$ is linearly independent over $\mathbb{Q}$. Then for each $\varepsilon > 0$ we have
$$
D_N( (n\boldsymbol{\gamma} \bmod [0,1)^s)_{n\ge 0} ) \ll N^{\varepsilon-1}.
$$
\end{lemma}
\begin{proof}
Using a classical result by Schmidt~\cite[Theorem~2]{Schmidt:70} we see that a vector $(\gamma_1,\ldots, \gamma_s)$ of real algebraic numbers for which $\{1,\gamma_1,\ldots, \gamma_s\}$ is linearly independent over $\Q$ is of finite type $1$. By Kuipers and Niederreiter~\cite[p.~132, Exercise 3.17]{Kuipers-Niederreiter:74} (or \cite[Theorem~1.80]{Drmota-Tichy:97}) this implies the result.
\end{proof}

After these preparations we turn to the proof of Theorem~\ref{thm:main}.

\begin{proof}[{Proof of Theorem~\ref{thm:main}}]
In view of Proposition~\ref{prop:product} we have to estimate the quantities $\delta_{k_1,\ldots, k_s}$. First note that by the definition of the linear conjugacy $Q$ in \eqref{eq:Q}, setting $\boldsymbol{\alpha}=(\alpha_{k_1},\ldots, \alpha_{k_s})$, we can write
\begin{align*}
\delta_{k_1,\ldots, k_s} &= \sup_{S_1 \in \mathcal{S}_{k_1}^{(m_1)},\ldots,S_s \in \mathcal{S}_{k_s}^{(m_s)}} \left|\frac1N
\sum_{n=0}^{N-1}\mathbf{1}_{S_1\times \cdots \times S_s}(R(n,\boldsymbol{\alpha})) - \prod_{i=1}^s\lambda_{\mathbf{v}_{m_i}}(S_i)
\right|
\\
& = 
\sup_{S_1 \in \mathcal{S}_{k_1}^{(m_1)},\ldots,S_s \in \mathcal{S}_{k_s}^{(m_s)}} 
\left|\frac1N
\sum_{n=0}^{N-1}\mathbf{1}_{Q(S_1\times \cdots \times S_s)}(QR(n,\boldsymbol{\alpha})) - \mu(Q(S_1\times\cdots\times S_s))
\right|,
\end{align*}
with $R(n)=R(n,\boldsymbol{\alpha})$ as defined in \eqref{eq:rn}, $QR(n,\boldsymbol{\alpha})$ as in \eqref{eq:qrn}, and $\mu$ the Lebesgue measure on $\R^{(m_1-1)+\cdots+(m_s-1)}$. Since $\{1,\varphi_{m_1},\ldots,\varphi_{m_1}^{m_1-1},\ldots, \varphi_{m_s},\ldots,\varphi_{m_s}^{m_s-1}\}$ is linearly independent over $\mathbb{Q}$, due to \eqref{eq:qrn} the rotation $QR(n)$ satisfies the assumptions of Lemma~\ref{lem:discrepancy} (the offset $\boldsymbol{\alpha}$ doesn't change the discrepancy significantly, see~\cite[Lemma~1.7]{Drmota-Tichy:97}), and we have 
\begin{equation}\label{eq:discEst}
D_N((QR(n))_{n\ge 0}) \ll N^{\eps/3-1}
\end{equation}
for $\eps > 0$ chosen as in the statement of the theorem.

We will now estimate the quantities $\displaystyle \delta_{k_1,\ldots, k_s}$ in terms of this discrepancy. To this end choose $M=\lfloor\frac{ \log_2 N}{(m_1-1)+\cdots+(m_s-1)}\rfloor$ and cover $Q(S_1\times \cdots \times S_s)$ for $S_1\times \cdots \times S_s \in \mathcal{S}_{k_1}^{(m_1)}\times\cdots\times \mathcal{S}_{k_s}^{(m_s)}$ by boxes $\{U_1,\ldots, U_r\}$ and $\{V_1,\ldots, V_{r'}\}$ as specified in Lemma~\ref{lem:boxdim}. Set now $\mathbf{S}=S_1\times \cdots \times S_s$. Then, using the triangle inequality and arguing in the same way as in \eqref{eq:gr1} and \eqref{eq:gr2}, we get
\begin{align*}
\left|\frac1N
\sum_{n=0}^{N-1}\mathbf{1}_{Q\mathbf{S}}(QR(n)) - \mu(Q\mathbf{S})
\right|  &\le \sum_{j=0}^r
\left|\frac1N
\sum_{n=0}^{N-1}\mathbf{1}_{U_j}(QR(n)) - \mu(U_j)
\right| \\
&\hskip 0.3cm + \sum_{j=0}^{r'}
\left|\frac1N
\sum_{n=0}^{N-1}\mathbf{1}_{V_j\cap Q\mathbf{S}}(QR(n)) - \mu(V_j\cap Q\mathbf{S})
\right| \\
&\le \sum_{j=0}^r
\left|\frac1N
\sum_{n=0}^{N-1}\mathbf{1}_{U_j}(QR(n)) - \mu(U_j)
\right| \\
&\hskip 0.3cm + \sum_{j=0}^{r'} \left(
\left|\frac1N
\sum_{n=0}^{N-1}\mathbf{1}_{V_j}(QR(n)) - \mu(V_j)
\right| 
+2\mu(V_j)
\right).
\end{align*}
Since $U_j$ and $V_j$ are boxes, the moduli on the right hand side can be estimated by the discrepancy of $QR(n)$ (again the fact that the boxes are not located at the origin doesn't cause significant difference by~\cite[Lemma~1.7]{Drmota-Tichy:97}).  Hence, taking the supremum over all $\mathbf{S}\in \mathcal{S}_{k_1}^{(m_1)}\times\cdots\times \mathcal{S}_{k_s}^{(m_s)}$ and keeping in mind that by Lemma~\ref{lem:boxdim} we have the estimates $r,r'\ll 2^{M(D+\varepsilon/3)}$ and 
$$
\sum_{j=0}^{r'}\mu(V_j) \ll 2^{M(D-(m_1-1)-\cdots-(m_s-1)+\varepsilon/3)}
$$
this yields
$$
\delta_{k_1,\ldots, k_s} \ll 2^{M(D+\varepsilon/3)} D_N((R(n))_{n\ge 0}) + 2^{M(D-(m_1-1)-\cdots-(m_s-1)+\varepsilon/3)}.
$$
By \eqref{eq:discEst} we finally end up with
\begin{align*}
\delta_{k_1,\ldots, k_s} &\ll
N^{\frac{D+\eps/3}{(m_1-1)+\cdots+(m_s-1)}-1+\frac\eps3} + N^{\frac{D-(m_1-1)+\cdots+(m_s-1)}{(m_1-1)+\cdots+(m_s-1)}+\frac\eps3}\ll N^{\frac{\max\{d_i-(m_i-1)\,:\,1\le i\le s\}}{(m_1-1)+\cdots + (m_s-1)}+\frac{2\varepsilon}3}.
\end{align*}
Inserting this in Proposition~\ref{prop:product} proves \eqref{eq:main}; indeed the sums over $k_1,\ldots, k_s$ occurring in \eqref{eq:prop42} contribute only logarithmic factors that are absorbed by $N^{\eps/3}$, and $\sum_{i=1}^s\beta_i^{-L_i} \ll N^{-1}$. 

The claim that $\mathrm{dim_B}(\partial \mathcal{R}_{m_i})<m_i-1$ has already been treated in the paragraph preceding~\eqref{eq:boxestimate}.
\end{proof}

We conclude this section with an easy example.

\begin{example}\label{ex:1}
Consider the golden mean $\varphi_2$ and the dominant root $\varphi_3$ of the {\em tribonacci polynomial} $X^3-X^2-X-1$. The Rauzy fractal $\mathcal{R}_2$ is an interval, hence, $\mathrm{dim_B}\partial \mathcal{R}_2=0$. For $\mathcal{R}_3$ we know from \cite{Ito-Kimura:91} that $\mathrm{dim_B}\partial \mathcal{R}_3=1.09336\ldots$ (in fact, in \cite{Ito-Kimura:91} the Hausdorff dimension is determined, however, for the set $\partial \mathcal{R}_3$ the Hausdorff dimension is the same as the box counting dimension because the restriction of $B_3$ to $\mathbf{v}_3^\bot$ is a similarity transformation). Since $\mathbb{Q}(\varphi_2,\varphi_3)$ has degree 6 over $\mathbb{Q}$ the linear independence assumption in Theorem~\ref{thm:main} is satisfied and we obtain (choosing $\varepsilon>0$ sufficiently small)
\[
D_N((H_{(\varphi_2,\varphi_3)}(n))_{n\ge 0}) \ll N^{-0.30221}.
\]
\end{example}	
	
\section{Final remarks on possible further research}
There are several directions for further research on this topic. The first task would be to generalize Theorem~\ref{thm:main} to the full class of linear recurrences given in \eqref{eq:Barat} or in \cite{Hofer-Iaco-Tichy:15}. To do this, several obstacles have to be mastered. The reason is that the dominant root $\beta$ of the characteristic polynomial of the recurrences in \eqref{eq:Barat} is a Pisot number but in general not a unit. Again one can associate substitutions to these linear recurrences, however, the Rauzy fractals no longer live in Euclidean space but in an open subring of the ad\`ele ring $\mathbb{A}_{\mathbb{Q}(\beta)}$. The theory of these fractals is well developed (see {\it e.g.}~\cite{BaratEtAl:08,Minervino-Steiner:14,Minervino-Thuswaldner:14,Siegel:03}), and with some more technical effort they should relate the Halton sequences in question with certain rotations on these subrings. To estimate the discrepancy of these rotations generalizations of the Erd\H{o}s-Tur\'an-Koksma inequality ({\it cf.}\ \cite[Theorem~1.21]{Drmota-Tichy:97} for the Euclidean version) and Schlickewei's $\mathfrak{p}$-adic subspace theorem ({\it cf.}~\cite{Schlickewei:77}) could be of use. We want to come back to this in a forthcoming paper.

To be more general one could define Halton sequences for the substitutive case (based on Dumont-Thomas numeration in the sense of Steiner~\cite{Steiner:09}), where even more examples with symmetric languages should come up.

Beyond that it would be interesting to get results on Halton sequences related to $\beta$-expansions with asymmetric languages. In this case it is not so clear how to proceed and one needs to deal with the reverse language in some way to define the appropriate Rauzy fractals in order to derive the rotation related to the Halton sequence in question. Also generalizations of Drmota's hybrid case (see~\cite{Drmota:15}) deserve interest. 

Since the discrepancy estimates in Theorem~\ref{thm:main} are certainly not optimal it would be of great interest to gain a better understanding of the distribution properties of $\boldsymbol{\beta}$-adic Halton sequences that would lead to improved discrepancy estimates and to the characterization of bounded remainder sets (see~\cite{Steiner:06} for bounded remainder sets for $\beta$-adic van der Corput sequences).

\bibliography{vdC}
\bibliographystyle{abbrv}

\end{document}